\documentclass[letterpaper,12pt]{amsart}
\usepackage{amsmath,amssymb,amsxtra,amsthm, amstext, amscd,amsfonts,fancyhdr, hyperref}
\newcommand{\bA}{{\mathbb{A}}}

\newcommand{\bC}{{\mathbb{C}}}

\newcommand{\bF}{{\mathbb{F}}}

\newcommand{\bN}{{\mathbb{N}}}

\newcommand{\bQ}{{\mathbb{Q}}}
\newcommand{\bR}{{\mathbb{R}}}

\newcommand{\bZ}{{\mathbb{Z}}}

\newcommand{\Ba}{{\mathbf{a}}}

\newcommand{\Bx}{{\mathbf{x}}}

  \newcommand{\B}{{\mathcal{B}}}
  
  \newcommand{\D}{{\mathcal{D}}}
  
  \newcommand{\F}{{\mathcal{F}}}

\renewcommand{\L}{{\mathcal{L}}}
  
  \newcommand{\N}{{\mathcal{N}}}

  \newcommand{\Q}{{\mathcal{Q}}}
  \newcommand{\R}{{\mathcal{R}}}
\renewcommand{\S}{{\mathcal{S}}}
  
  \newcommand{\U}{{\mathcal{U}}}

\newcommand{\Beu}{\operatorname{Beu}}

\newcommand{\ep}{\varepsilon}

\newcommand{\ol}{\overline}

\newcommand{\upchi}{{\raise.35ex\hbox{$\chi$}}}

\newcommand{\bup}{\boldsymbol \upsilon} 
\newcommand{\Gal}{\operatorname{Gal}}

\newcommand{\fd}{\mathfrak{d}}

\makeatletter
\@namedef{subjclassname@2010}{%
  \textup{2010} Mathematics Subject Classification}
\makeatother

\newtheorem{theorem}{Theorem}[section]

\newtheorem{proposition}[theorem]{Proposition}
\newtheorem{lemma}[theorem]{Lemma}

\theoremstyle{definition}

\numberwithin{equation}{section}

\begin{document}

\title{Square-free values of decomposable forms}
\author{Stanley Yao Xiao}
\address{Mathematical Institute, University of Oxford, Oxford, UK}
\date{\today}


\begin{abstract} In this paper we prove that decomposable forms, or homogeneous polynomials $F(x_1, \cdots, x_n)$ with integer coefficients which split completely into linear factors over $\mathbb{C}$, take on infinitely many square-free values subject to simple necessary conditions and $\deg f \leq 2n + 2$ for all irreducible factors $f$ of $F$. This work generalizes a theorem of Greaves.
\end{abstract}

\maketitle


\section{Introduction}

In this paper, we consider the density of integer tuples $(x_1, \cdots, x_n)$ satisfying $|x_i| \leq B$ and for which $F(x_1, \cdots, x_n)$ is square-free, where $F$ is an $n$-ary \emph{decomposable form} of degree $d > n$. A homogeneous polynomial $F$ is said to be a decomposable form if it splits into linear factors over the algebraic closure of its field of definition. If $F$ has rational coefficients and is irreducible over $\bQ$, we say that $F$ is an \emph{incomplete norm form}. Before stating our result, we shall give a brief summary of work done on square-free values of polynomials to date. \\

For a polynomial $g(x)$ with integer coefficients, define the counting function
\[N_g(B) = \# \{x \in \bZ : |x| \leq B, g(x) \text{ is square-free}\}.\]
Estermann \cite{Est} showed that when $g(x) = x^2 + 1$, there exists a positive number $c_g$ such that the asymptotic formula
\begin{equation} \label{1 var} N_g(B) = c_g B + O(B^{2/3} \log B)\end{equation}
holds. We will say that a polynomial $g$ has \emph{no fixed square divisor} if for all primes $p$ there exists $n_p \in \bZ$ such that $p^2 \nmid g(n_p)$. Ricci \cite{Ric} generalized Estermann's work and showed that for any irreducible quadratic polynomial with no fixed square divisor, there exists a positive number $c_g$ such that (\ref{1 var}) holds. Erd\H{o}s showed that 
\[\lim_{B \rightarrow \infty} N_g(B) = \infty\]
in \cite{E} for cubic polynomials with no fixed square divisor. Hooley \cite{Hoo} refined the work of Estermann, Ricci, and Erd\H{o}s and showed that for all cubic polynomials $g$ with no fixed square divisor, there exists a positive number $c_g$ such that (\ref{1 var}) holds with a worse error term. Helfgott further refined Hooley's work in \cite{Hel} by showing that an analogous asymptotic formula to (\ref{1 var}) holds when we replace integer inputs with prime inputs. To date, it is not known whether (\ref{1 var}) holds unconditionally for any polynomial $g$ with no fixed square divisor with $\deg g \geq 4$. \\

Assuming the $abc$-conjecture, Granville and Poonen proved respectively in \cite{Gran} and \cite{Poo} that polynomials in a single variable and polynomials in multiple variables take on infinitely many square-free values. We note that Poonen's result does not allow one to deduce an analogous asymptotic formula to (\ref{1 var}). Bhargava, Shankar, and Wang recently showed the existence of an asymptotic formula for square-free values of \emph{discriminant polynomials}, which does not use the $abc$-conjecture in \cite{BSW}. \\

A natural generalization from the case of single-variable polynomials is to binary forms. Greaves made a breakthrough in \cite{Gre} on the problem of square-free values of binary forms for suitable binary forms $F(x,y)$ with integer coefficients with no fixed square divisor. He showed that the density of integer pairs $(x,y)$ such that $F(x,y)$ is square-free is exactly as expected provided that $d' \leq 6$, where $d'$ is the largest degree of an irreducible factor of $F$. One observes that the requirement $d' \leq 6$ can be compared to $d \leq 3$ in the single variable case.  Hooley, in \cite{Hoo1} and \cite{Hoo2}, extended Greaves's results to the case when $F$ is a polynomial in two variables which splits into linear factors over $\bC$. \\ 

Schmidt, in \cite{Sch}, introduced an invariant which he called the discriminant for (incomplete) \emph{norm forms} which we define below. Write
 \begin{equation} \label{norm factor} F(\Bx) = \prod_{j=1}^d L_j(\Bx),\end{equation}
where the $L_j$'s are conjugates of the linear form
\[L_1(\Bx) = \omega_1 x_1 + \omega_2 x_2 + \cdots + \omega_n x_n\]
with algebraic integer coefficients in a number field $K$. We then put
\begin{equation} \label{discriminant} \Delta(F) = \prod_{\{i_1, \cdots, i_n\} \subset \{1, \cdots, d\}} \lvert \det(L_{i_1}, \cdots, L_{i_n})\rvert, \end{equation}
where the determinant of $n$ linear forms in $x_1, \cdots, x_n$ refers to the determinant of its coefficients. It is easy to check that $\Delta(F)$ is invariant under any action of the Galois group $\Gal(\ol{\bQ}/\bQ)$, and since each term that appears in the product is an algebraic integer, it follows that $\Delta(F)$ is a rational integer. We say that $F$ has \emph{bad reduction} at a prime $p$ if $F$ has a repeated linear factor over $\bF_p$. One notes that bad reduction can only occur if $p | \Delta(F)$. Therefore, if $\Delta(F)$ is non-zero, then bad reduction can only occur at finitely many primes. \\ 

In this paper, we extend Greaves's work in \cite{Gre} and Hooley's work in \cite{Hoo1} and \cite{Hoo2} by generalizing Greaves's geometry of numbers method for $n$-ary \emph{decomposable forms} and adapting Hooley's sieve arguments. \\ 

For an integer $k$ and an integer $m$, we say that $m$ is $k$-free if for all primes $p$ dividing $m$, we have $p^k \nmid m$. For a set $S$, we write $\# S$ for the cardinality of $S$. Let us write, for an $n$-ary form $F$ with integer coefficients,
\begin{equation} \label{rho} \rho_F(m) = \# \{(a_1, \cdots, a_n) \in (\bZ/m\bZ)^n : F(a_1, \cdots, a_n) \equiv 0 \pmod{m}\}\end{equation}
and for a positive number $B$ and an integer $k \geq 2$,
\begin{equation} \label{count} N_{F,k}(B) = \# \{(x_1, \cdots, x_n) \in \bZ^n : |x_i| \leq B, F(x_1, \cdots, x_n) \text{ is } k\text{-free}\}.\end{equation}
We will prove the following theorem:
\begin{theorem} \label{MT1} Write $\Bx = (x_1, \cdots, x_n)$ and let 
\[F(\Bx) = L_1(\Bx) \cdots L_r(\Bx)\] 
be a \emph{decomposable} form with integer coefficients and non-zero discriminant $\Delta(F)$ as given in (\ref{discriminant}), where $L_1, \cdots, L_r$ are linear forms with algebraic integral coefficients in some finite extension $K/\bQ$. Let $d$ be the maximal degree of a $\bQ$-irreducible factor of $F$. Let $k \geq 2$ be an integer with the property that for all primes $p$, there exists a vector $\Bx^{(p)} = \left(x_1^{(p)}, \cdots, x_n^{(p)}\right) \in \bZ^n$ such that $p^k \nmid F\left(\Bx^{(p)}\right)$. Then the asymptotic relation 
\[N_{F,k}(B) \sim B^n \prod_p \left(1 - \frac{\rho_F(p^k)}{p^{nk}} \right) \]
holds whenever 
\begin{equation} \label{k-d rel}k \geq \frac{d-2}{n}.\end{equation}
\end{theorem} 
In particular, if $k = 2$, then $F$ takes on infinitely many square-free values as long as $d \leq 2n+2$. This recovers the theorem of Greaves in \cite{Gre}. We further remark that J.~Maynard, in \cite{May}, used methods from geometry of numbers related to the methods in Section \ref{Geom}, to prove an analogous theorem to Theorem \ref{MT1} for primes represented by incomplete norm forms. \\

The outline of our paper is as follows. In Section \ref{Prelim} we will use an elementary sieve argument to partition the relevant main terms and error terms to be estimated in order to prove Theorem \ref{MT1}. In Section \ref{Geom}, we will generalize Greaves's geometry of numbers argument in \cite{Gre} to the case of decomposable forms over $\bZ$. In Sections \ref{Ekedahl} and \ref{Selberg}, we adapt the Ekedahl Sieve as described in \cite{BhaSha} and \cite{Eke} and the Selberg sieve, as expressed by Hooley in \cite{Hoo1}, to establish an estimate for the remaining error terms relevant to condition (\ref{k-d rel}) of Theorem \ref{MT1}. 

\subsection*{Funding} This work was supported by the Government of Ontario; University of Waterloo; and the Natural Sciences and Engineering Research Council of Canada.  

\subsection*{Acknowledgments} The author would like to thank C.L. Stewart and D. Tweedle for their comments on the paper and discussions. 

\section{Preliminaries}
\label{Prelim}

We will show that $N_{F,k}(B)$ (recall (\ref{count})) satisfies an inequality of the form
\begin{equation} \label{simple sieve} N_1(B) - N_2(B) - N_3(B) \leq N_{F,k}(B) \leq N_1(B).
\end{equation} 
Our goal will be to demonstrate that for any $\ep > 0$, that
\[N_1(B) = B^n \prod_{p \leq \xi_1} \left(1 - \frac{\rho_F(p^k)}{p^{nk}}\right) + O_{F,\ep} \left(B^{n - 1 + \ep}\right),\]
and for some $\delta_n > 0$ and some slowly growing function $\xi_1 = \xi_1(B)$ tending to infinity as the parameter $B$ tends to infinity, that
\[N_2(B) = O_F\left(B^n \left(\xi_1^{-1} + (\log B)^{-\delta_n} \right) \right)\]
and that
\[N_3(B) = o_F(B^n).\]
Put $\log_1 (B) = \max\{1, \log B\}$ and $\log_s B = \log_1 \log_{s-1} B$ for $s \geq 2$. We now write
\begin{equation} \label{xi1} \xi_1 = \xi_1(B), \end{equation}
to be an eventually increasing real-valued function tending to infinity which we shall define later. For now, it suffices to suppose that $\xi_1(B) = O(\log_2 B/\log_3 B)$. Next put
\begin{equation} \label{xi2} \xi_2 = B^n (\log B)^{2/3}.\end{equation}
Now define
\begin{equation} \label{M1} N_1(B) = \# \{\Bx \in \bZ^n : |x_i| \leq B, \text{if } p^k | F(\Bx), \text{then } p > \xi_1 \},\end{equation}
\begin{equation} \label{M2} N_2(B) = \# \{\Bx \in \bZ^n : |x_i| \leq B, \text{there exists } p \in (\xi_1, \xi_2] \text{ s.t. } p^2 | F(\Bx), \text{and} \end{equation}
\[\text{if } p^k | F(\Bx), \text{then } p > \xi_1\},\]
and
\begin{equation} \label{M3} N_3(B) = \# \{\Bx \in \bZ^n : |x_i| \leq B, \text{there exists } p > \xi_2 \text{ s.t. } p^k | F(\Bx), F(\Bx) \text{ is indivisible by }  \end{equation}
\[p^2 \text{ for } \xi_1 < p \leq \xi_2 \text{ and } \text{if } p^k | F(\Bx), \text{then } p > \xi_1\}.\]
Before we proceed with estimating $N_1(B)$, let us establish some facts about the function $\rho_F$ as defined in (\ref{rho}). For a positive integer $m$ and a real number $\alpha$, let us write
\[\sigma_{\alpha}(m) = \sum_{s | m} s^{\alpha}.\]
Furthermore, for each prime $p$ we define
\begin{equation} \label{tau} \tau_F(p) = \# \text{ geometrically irreducible components of } F \text{ defined over } \bF_p, \end{equation}
and for square-free integers we define
\[\tau_F(m) = \prod_{p | m} \tau_F(p).\]
We remark that in our case, the only geometrically irreducible components are hyperplanes which are defined over $\bF_p$. \\

We will establish the following lemma:
\begin{lemma} \label{rho lemma} Let $\rho_F$ be defined as in $(\ref{rho})$. Then $\rho_F$ is multiplicative and for all primes $p$ we have
\[\rho_F(p^k) = O_{d,n} \left(p^{k(n-1)} + p^{n(k-1)}\right).\]
If $m$ is a square-free integer, then
\[\rho_F(m) = O_F(m^{n-1} \tau_F(m) \sigma_{-1/4}(m)).\]
\end{lemma}
\begin{proof} The fact that $\rho_F$ is multiplicative follows from the Chinese Remainder Theorem. For the upper bound, let us first suppose that there exists an index, say $i = 1$, such that $p \nmid x_1$. Then there are at most $p^k$ many choices for $x_2, \cdots, x_n$. Having fixed these, there are then at most $d$ choices for $x_1$. Hence, there are at most $ndp^{(n-1)k}$ choices for $(x_1, \cdots, x_n)$. Otherwise, suppose that $p | x_i$ for $i = 1, \cdots, n$. Write $x_i = p x_i'$ for $i = 1, \cdots, n$. Then there are at most $p^{k-1}$ choices for each $i = 1, \cdots, n$, whence there are $p^{n(k-1)}$ choices altogether. Combining these, we obtain the claimed upper bound. \\ \\
For the second part, we use a result of Lang-Weil in \cite{LW}, which asserts that for any algebraic variety $V$ defined over $\bQ$ and any prime $p$, we have
\begin{equation} \label{Lang-Weil} \# V(\bF_p) = C_V(p) p^{\dim V} + O_V \left(p^{\dim V - 1/2}\right), \end{equation} 
where $C_V(p)$ is the number of geometrically irreducible, top-dimensional components of $V$ which are defined over $\bF_p$. We then have
\[\rho_F(p) = \tau_F(p) p^{n-1} + O_F(p^{n-3/2}).\]
Multiplicativity of $\rho_F$ then yields 
\begin{align*} \rho_F(m) & =  \prod_{p | m} \left(\tau_F(p) p^{n-1} + O_F(p^{n-3/2})\right) \\
& = m^{n-1} \prod_{p | m} \left(\tau_F(p)  + O_F(p^{-1/2})\right) \\ 
& = O_F(m^{n-1}\tau_F(m) \sigma_{-1/4}(m)).\end{align*} \end{proof} 
We remark that Lemma \ref{rho lemma} implies that the infinite product
\[\prod_p \left(1 - \frac{\rho_F(p^k)}{p^{nk}} \right)\]
converges. This is because 
\[\frac{\rho_F(p^k)}{p^{nk}} = O \left(\frac{1}{p^k} + \frac{1}{p^n}\right) = O \left(\frac{1}{p^2}\right),\]
since $k, n \geq 2$ by assumption. \\

We give an estimate for $N_1(B)$. Define, for a positive integer $b$, the quantity
\[N(b,B) = \# \{\Bx \in \bZ^n \cap [-B,B]^n : b^k | F(\Bx)\}.\]
Then from the familiar property of the Mobius function $\mu$, we have
\begin{align*} N_1(B) & = \sum_{\substack{ b \in \bN \\ p | b \Rightarrow p \leq \xi_1}} \mu(b) N(b,B) \\
& = \sum_{\substack{ b \in \bN \\ p | b \Rightarrow p \leq \xi_1}} \mu(b) \rho_F(b^k) \left(\frac{B^n}{b^{nk}} + O\left(\frac{B^{n-1}}{b^{(n-1)k}} + 1\right) \right) \\
& = B^n \prod_{p \leq \xi_1} \left(1 - \frac{\rho_F(p^k)}{p^{nk}}\right) + O\left(\sum_{\substack{ b \in \bN \\ p | b \Rightarrow p \leq \xi_1}} \rho_F(b^k) \left(\frac{B^{n-1}}{b^{(n-1)k}} + 1 \right) \right).
\end{align*}
By the theorem of Rosser and Schoenfeld \cite{RS}, it follows that for all $\ep > 0$ and some $C' > 0$ we have
\[\prod_{p \leq \xi_1} p \leq e^{2 \xi_1} = O \left((\log B)^{\frac{C'}{\log_3 B}} \right) = O_\ep (B^\ep),\]
by (\ref{xi1}). Hence, we obtain via Lemma \ref{rho lemma} that, for any $\ep > 0$,
\[N_1(B) = B^n \prod_{p \leq \xi_1} \left(1 - \frac{\rho_F(p^k)}{p^{nk}}\right) + O\left(\sum_{b \ll_\ep B^\ep} B^{n-1+\ep} + b^{n(k-1)+\ep} + b^{k(n-1)+\ep} \right).\]
We then see that
\begin{equation} \label{N1B main estimate} N_1(B) = B^n \prod_{p \leq \xi_1} \left(1 - \frac{\rho_F(p^k)}{p^{nk}}\right) +  O_{\ep} \left(B^{n-1 + \ep}  \right).\end{equation}
As $B \rightarrow \infty$, the partial product in (\ref{N1B main estimate}) tends to the convergent product in Theorem \ref{MT1}, thus it suffices to show that $N_2(B), N_3(B)$ are error terms. \\ 

In the next section we will see that we can obtain good estimates for $N_2(B)$ even when $\xi_2$ is as large as $B^n (\log B)^{2/3}$. Let 
\[F(x_1, \cdots, x_n) = \F_1(\Bx) \cdots \F_r (\Bx),\]
where each $\F_i$ is irreducible over $\bQ$ for $i = 1, \cdots, r$. Here $d = \max_{1 \leq j \leq r} \deg \F_j$. Let us write 
\[N_2^{(j)}(B) = \# \{\Bx \in \bZ^n : |x_i| \leq B, \text{there exists } p \in (\xi_1, \xi_2] \text{ s.t. } p^k | \F_j(\Bx), \text{ and}\]
\[\text{if } p^k | \F_j(\Bx), \text{then } p > \xi_1\},\]
and
\[N_3^{(j)}(B) = \# \{\Bx \in \bZ^n : |x_i| \leq B, \text{there exists } p > \xi_2 \text{ s.t. } p^k | \F_j(\Bx),  \]
\[p^2 \nmid F_j(\Bx) \text{ for } \xi_1 < p \leq \xi_2, \text{ and if } p^k | \F_j(\Bx), \text{then } p > \xi_2\}.\]
If $\Bx$ is counted by $N_2(B)$ (respectively $N_3(B)$) but not by $N_2^{(j)}(B)$ (respectively $N_3^{(j)}(B)$) for $j =1 ,\cdots, r$, then there must exist $j_1 < j_2$ and a positive integer $k' < k$ such that 
\[\F_{j_1}(\Bx) \equiv 0 \pmod{p^{k'}} \text{ and } \F_{j_2}(\Bx) \equiv 0 \pmod{p^{k-k'}}.\]
However, this can only happen if $p | \Delta(F)$, so this situation can be avoided if $B$ is chosen sufficiently large. Hence, we have
\[N_2(B) \leq \sum_{j=1}^r N_2^{(j)}(B)\]
and
\[N_3(B) \leq \sum_{j=1}^r N_3^{(j)}(B).\]
It therefore suffices to deal with the case when $F$ is irreducible over $\bQ$ and $d = \deg F$. 

\section{Geometry of Numbers}
\label{Geom}

In this section we shall give an estimate for $N_2(B)$. To do so, we show that for each modulus $m$ we can reduce the problem to counting integer points of bounded height in a finite number $\N_F$ of lattices, the important feature being that $\N_F$ is dependent only on $F$. 

\begin{lemma} \label{lattice lemma 0} Let $F \in \bZ[x_1, \cdots, x_n]$ be an incomplete norm form of degree $d > n$. Let $p \nmid \Delta(F)$ be a prime, and let $\Ba = (a_1, \cdots, a_n) \in \bZ^n$ be a solution to the congruence
\[F(\Bx) \equiv 0 \pmod{p^2}.\]
Then $\Ba$ lies on a finite number $\N_F$ of lattices $\Lambda \subset \bZ^n$. Moreover, for each such lattice $\Lambda$, we have $\det \Lambda \geq p^2$. 
\end{lemma}

\begin{proof} 
By the same argument as that in Section 5 of \cite{Hoo2}, we can factor $F$ into 
\[F(\Bx) = F^\ast(\Bx)\prod_{i=1}^{\tau_F(p)} \L_i(\Bx) ,\]
where $\L_i(\Bx) = \upsilon_1^{(i)} x_1 + \upsilon_2^{(i)}x_2 + \cdots + \upsilon_n^{(i)} x_n$ are defined over $\bZ_p$, while $F^\ast(\Bx)$ is a form defined over $\bZ_p$. Suppose that $\Ba = (a_1, \cdots, a_n) \in \bZ^n$ is a solution to the congruence
\[F(\Ba) \equiv 0 \pmod{p^2}.\]
Then $\Ba$ is of one of the following types:
\begin{itemize}
\item[(a)] There exists exactly one $i, 1 \leq i \leq \tau_F(p)$ such that $\L_i(\Ba) \equiv 0 \pmod{p^2}$, while $\L_j(\Ba) \not\equiv 0 \pmod{p}$ for $j \ne i$, and $F^\ast(\Ba) \not \equiv 0 \pmod{p}$. 
\item[(b)] There exist $1 \leq i_1 < i_2 \leq \tau_F(p)$ such that 
\[\L_{i_1}(\Ba) \equiv \L_{i_2}(\Ba) \equiv 0 \pmod{p}.\]
\item[(c)] $F^\ast(\Ba) \equiv 0 \pmod{p}$. 
\end{itemize}
If $\Ba$ is of type (a), then $\Ba$ lies in the union of at most $\tau_F(p) \leq d$ lattices of determinant $p^2$. If $\Ba$ is of type (b), then there are two further sub-cases. Firstly, and more simply, there exist two indices $i_1 < i_2$ and an integer $t$ such that 
\begin{equation} \label{bad red}\L_{i_1}(\Bx) \equiv t \L_{i_2}(\Bx) \pmod{p}.\end{equation}
If (\ref{bad red}) holds, then it follows that $\Delta(F) \equiv 0 \pmod{p}$, hence $p$ divides the discriminant $\Delta(F)$ of $F$. Thus, there are only finitely many primes for which this could happen. Otherwise, $\Ba$ lies on the intersection of two distinct lattices $\Lambda_1, \Lambda_2$ of determinant $p$, defined by
\[\Lambda_1 = \{\Bx \in \bZ^n : \Bx \cdot \Ba_1 \equiv 0 \pmod{p}\}\]
and
\[\Lambda_2 = \{\Bx \in \bZ^n : \Bx \cdot \Ba_2 \equiv 0 \pmod{p}\},\]
where $\Ba_1, \Ba_2$ are two non-proportional non-zero vectors modulo $p$. Now let $\phi_1, \phi_2$ be homomorphisms from $\bZ^n$ to $\bF_p$ defined by 
\[\phi_1(\Bx) = \Ba_1 \cdot \Bx \pmod{p}\]
and
\[\phi_2(\Bx) = \Ba_2 \cdot \Bx \pmod{p}.\]
Then $\Lambda_1, \Lambda_2$ are the kernels of $\phi_1, \phi_2$ respectively. Now let $\phi$ be defined by $\phi : \bZ^n \rightarrow (\bZ/p\bZ)^2, \phi(\Bx) = (\phi_1(\Bx), \phi_2(\Bx))$. The image of $\phi$ is the full set $(\bZ/p\bZ)^2$ whenever $\Ba_1, \Ba_2$ are not proportional modulo $p$. Hence, $\Ba$ lies in a lattice of determinant at least $p^2$. Further, there are at most $\tau_F(p)^2 \leq d^2$ such lattices. \\ \\
If $\Ba$ is of type (c), then modulo $p$ there exists a linear factor $\L_j$ of $F^\ast$ which is not defined over $\bF_p$ such that $\L_j(\Ba) \equiv 0 \pmod{p}$. Let $s$ be the degree of the field of definition of $\L_j$ over $\bF_p$. By assumption, we have $s \geq 2$. Then $\L_j$ can be written as
\[\L_j = \alpha_1 \L_{j,1} + \cdots + \alpha_s \L_{j,s},\]
where $\L_{j,i}$ are linear forms with coefficients in $\bF_p$ and $\alpha_1, \cdots, \alpha_s$ is a basis of $\bF_{p^s}$ over $\bF_p$. In particular, $\alpha_1, \cdots, \alpha_s$ are linearly independent over $\bF_p$. Therefore, $\L_j(\Ba) \equiv 0 \pmod{p}$ implies that $\L_{j,i}(\Ba) \equiv 0 \pmod{p}$ for $i = 1, \cdots, s$. It thus follows that $\Ba$ lies in the intersection of the lattice in $\bZ^n$ given by the linear forms $\L_{j,1}, \L_{j,2}$, hence by the same argument it follows that $\Ba$ lies in a lattice of determinant at least $p^2$. Moreover, the number of such lattices is at most $d^2$. \end{proof}

Now we generalize Lemma 1 in \cite{Gre} (see also \cite{HB0}) for norm forms in $n \geq 2$ variables. Indeed, we will prove the following:

\begin{lemma} \label{lattice lemma} Let $\Lambda \subset \bZ^n$ be a lattice of determinant $m$. For $\Bx \in \bZ^n$ denote by $H(\Bx)$ the sup norm of $\Bx$. Put
\[N_\Lambda(B ) = \{\Bx \in \bZ^n : H(\Bx) \leq B\}\]
and put $M_\Lambda$ for the sup norm of the shortest vector in $\Lambda$. Then 
\[N_\Lambda(B) \ll_n \frac{B^n}{m} + O \left(\frac{B^{n-1}}{M_\Lambda^{n-1}} + 1 \right).\]
\end{lemma}

\begin{proof} Let $\Bx_1 = \left(x_1^{(1)}, \cdots, x_n^{(1)}\right)$ be one of the shortest vectors with respect to sup norm. Without loss of generality, we may assume that $\lvert x_1^{(0)} \rvert = M_\Lambda$. Observe that $M_\Lambda \leq m^{1/n}$. To see this, let $l = l(m)$ denote the smallest positive integer such that $(l + 1)^n > m$. Then there exist two distinct vectors $\Ba_1, \Ba_2$ such that the coordinates of both vectors are at most $l/2$ in absolute value and
\[\Ba_1 \equiv \Ba_2 \pmod{m},\]
whence their difference $\Ba_1 - \Ba_2$ lies in $\L$ and $H(\Ba_1 - \Ba_2) \leq m^{1/n}$. \\ \\
By Lemma 4.3 in \cite{B1}, there exist vectors $\Bx_2, \cdots, \Bx_n \in \L$ such that
\[m \leq \prod_{j=1}^n H(\Bx_j) \ll_n m,\]
and for all vectors $\Bx \in \L$, if we write
\[\Bx = \sum_{j=1}^n \lambda_j \Bx_j,\]
we have 
\[|\lambda_j| \ll_n \frac{H(\Bx)}{H(\Bx_j)}.\]
In particular, for a vector $\Bx$ counted by $N_{\Lambda}(B)$, we have
\[|\lambda_j| \ll_n \frac{B}{H(\Bx_j)}. \]
By observing that $H(\Bx_j) \geq M_\Lambda$ for $j = 1, \cdots, n$, we obtain the bound
\begin{align*} N_{\Lambda}(B) & \ll_n \prod_{j=1}^n \left(1 + \frac{B}{H(\Bx_j)} \right) \\ 
& \ll_n \frac{B^n}{m} + \frac{B^{n-1}}{M_\Lambda^{n-1}} + \cdots + 1.
\end{align*}
Hence we obtain the consequence of the lemma. \end{proof}

For each prime $p$, we denote by $\U_p$ the set of lattices containing the solutions to the congruence $F(\Bx) \equiv 0 \pmod{p^2}$. For each $\Lambda \in \U_p$, we say that $\Lambda$ is of type a), b), or c) if $\Lambda$ arises from a solution $\Ba$ to $F(\Bx) \equiv 0 \pmod{p^2}$ of type a), b), or c) in the proof of Lemma \ref{lattice lemma 0}. Then write $F_\Lambda$ to be equal to:
\begin{itemize}
\item[(a)] $L_i(\Bx)$, if $\Lambda$ is of type a) and $\L_i$ is the unique linear form associated to $\Lambda$;
\item[(b)] $\L_{i_1} \cdots \L_{i_s}$, where $\L_{i_1}, \cdots, \L_{i_s}$ are the linear factors of $F$ defined over $\bF_p$ which vanish on $\Lambda$ modulo $p$ when $\Lambda$ is of type b); and
\item[(c)] $F^\ast$ if $\Lambda$ is of type c). 
\end{itemize}
We now estimate $N_2(B)$ via the following lemma:

\begin{lemma} \label{Main Lem} The error term $N_2(B)$ satisfies
\[N_2(B) = O_n\left(B^n \left(\xi_1^{-1} + (\log B)^{-1/3n} \right) \right). \]
\end{lemma}

\begin{proof} Let $\U_p$ denote the set of at most $\N_F$ many lattices $\Lambda$, each with determinant at least $p^2$ by Lemma \ref{lattice lemma 0}, which contains all of the solutions to $F(\Bx) \equiv 0 \pmod{p^2}$. Then
\[N_2(B) \ll_n \sum_{\xi_1 < p \leq B^n (\log B)^{2/3}} \sum_{\Lambda \in \U_p} N_{\Lambda} (B). \]
By Lemma \ref{lattice lemma}, it follows that
\[N_2(B) \ll_n \sum_{\xi_1 < p \leq B^n (\log B)^{2/3}} \sum_{\Lambda \in \U_p} \left(\frac{B^n}{p^2} + \frac{B^{n-1}}{M_\Lambda^{n-1}} + 1 \right). \]
We first consider consider the term
\begin{equation} \label{M3B main error}\sum_{\xi_1 < p \leq \xi_2} \sum_{\Lambda \in \U_p} \frac{B^n}{p^2}.\end{equation}
The sum
\[\sum_{p> \xi_1} \sum_{1 \leq j \leq \N_F} \frac{1}{p^2}\]
converges and is bounded by $O_F\left(\xi_1^{-1}\right)$. Now we look at the sum
\[\sum_{\xi_1 < p \leq \xi_2} \sum_{\Lambda \in \U_p}  \frac{B^{n-1}}{M_\Lambda^{n-1}}. \]
We break the above sum into three sub-sums $S_1, S_2,$ and $S_3$. $S_1$ will consist of the contribution from those primes $\xi_1 < p \leq B$. In this case, we have
\begin{align*}S_1 & = \sum_{\xi_1 < p \leq B} \sum_{\Lambda \in \U_p} \frac{B^{n-1}}{M_\Lambda^{n-1}} \\
& \ll B^{n-1} \sum_{1 \leq j \leq \N_F} \sum_{p \leq B} 1 \\
& \ll \frac{B^n}{\log B},
\end{align*}
where we used the trivial estimate that $M_\Lambda \geq 1$. \\ \\
$S_2$ will be the sub-sum consisting of those $M_\Lambda \geq B (\log B)^{-1/3n}.$ In this case, we have
\begin{align*} S_2 & \ll_d \sum_{\xi_1 < p \leq B^n (\log B)^{2/3}} \sum_{\Lambda \in \U_p} \frac{B^{n-1}(\log B)^{(n-1)/3n}}{B^{n-1}} \\
& \ll_d (\log B)^{\frac{(n-1)}{3n}} \frac{B^n (\log B)^{2/3}}{\log B} \\
& \ll_d B^n (\log B)^{-1/3n}.
\end{align*} 
Finally, $S_3$ will denote the sub-sum consisting of those primes $p > B$ and $M_\Lambda \leq B(\log B)^{-1/3n}.$ We then have
\begin{align*} S_3 & \ll \sum_{0 < |x_1^{(1)}|, \cdots, |x_n^{(1)}| \leq B(\log B)^{-1/3n}} \sum_{M_\Lambda \in \U_p} \sum_{\substack{p^2 | F_\Lambda(\Bx_1) \\ p > B}} \frac{B^{n-1}}{M_\Lambda^{n-1}} \\
& \ll B^{n-1} \sum_{0 < \lvert x_1^{(1)} \rvert \leq B(\log B)^{-1/3n}} \frac{1}{\lvert x_1^{(1)}\rvert^{n-1}} \sum_{0 \leq \lvert x_2^{(1)} \rvert, \cdots, \lvert x_n^{(1)} \rvert \leq \lvert x_1^{(1)} \rvert} \sum_{\substack{p^2 | F(\Bx_1) \\ p > B}} 1 \\
& \ll B^{n-1} B (\log B)^{-1/3n},
\end{align*} 
the last inequality following form the fact that at most $\lfloor d/2 \rfloor + 1$ many primes with $p > B$ can satisfy $p^2 | F(\Bx_1)$, since $\lVert \Bx_1 \rVert \leq B$. \\ \\
Finally, the last term needing to be estimated is
\[\sum_{\xi_1 < p \leq B^n (\log B)^{2/3}} \sum_{\Lambda \in \U_p} 1.\]
This is bounded by the number of primes in the interval $[\xi_1, B^n (\log B)^{2/3}]$, which by the prime number theorem is $O(B^n (\log B)^{2/3}/\log B) = O(B^n (\log B)^{-1/3})$, and so constitutes a negligible error term. 
\end{proof} 

\section{The Ekedahl sieve}
\label{Ekedahl}

In this section, we use the following result of Ekedahl in \cite{Eke} to handle certain contributions to $N_3(B)$. The version below was formulated by Bhargava and Shankar in \cite{BhaSha}: 

\begin{proposition}[Ekedahl sieve] \label{Ekedahl sieve} Let $\B$ be a compact region in $\bR^n$ having finite measure, and let $Y$ be any closed subscheme of $\bA_\bZ^n$ of co-dimension $s \geq 2$. Let $r$ and $M$ be positive real numbers. Then we have 
\[\#\{\Bx \in r \B \cap \bZ^n : \Bx \pmod{p} \in Y(\bF_p) \text{ for some prime } p > M\} \]
\[ = O \left(\frac{r^n}{M^{s-1} \log M} + r^{n - s + 1} \right).\]
\end{proposition}

We factor $F$ into linear factors over $\ol{\bQ}$, where
\begin{equation} \label{C factor} F(\Bx) = \prod_{j=1}^d \left(\psi_1^{(j)} x_1 + \cdots + \psi_n^{(j)} x_n \right) = \prod_{i=1}^d L_i(\Bx). \end{equation}
Let $Y_{i,j}$ denote the variety defined by $L_i(\Bx) = L_j(\Bx) = 0$, and let $Y = \bigcup_{1 \leq i < j \leq n} Y_{i,j}$. Since $Y$ is invariant under the action of $\Gal(\ol{\bQ}/\bQ)$, it is defined over $\bQ$. Moreover it has co-dimension at least two in $\bA_\bZ^n$. Let $p$ be a prime. Over $\bZ_p$, we have the factorization (see \cite{Hoo1}) of $F$ into
\[F(\Bx) = F^\ast(\Bx) \prod_{i=1}^{\tau_F(p)} \L_i(\Bx),\]
where $F^\ast, \L_i$ have $\bZ_p$-coefficients and $F^\ast$ does not have linear factors over $\bQ_p$. Let $\S_p$ be those congruence classes $\Bx$ in $(\bZ/p\bZ)^n = \bF_p^n$ such that either
\begin{itemize}
\item[(a)] There exist $1 \leq i < j \leq \tau_F(p)$ such that $\L_i(\Bx) \equiv \L_j(\Bx) \equiv 0 \pmod{p}$, or
\item[(b)] $F^\ast(\Bx) \equiv 0 \pmod{p}$. 
\end{itemize}
Since linear factors of $F^\ast$ are not defined over $\bF_p$ and hence has a non-trivial conjugate, it follows that whenever $\Ba \in \S_p$ that $\Ba \in Y(\bF_p)$. We then have the following consequence of Ekedahl's sieve:

\begin{lemma} \label{eke consequence} Let $N_3^\ast(B)$ denote the number of elements $\Bx \in \bZ^n \cap [-B,B]^n$ for which $\Bx \pmod{p} \in \S_p$ for some $p > \xi_1$. Then
\[N_3^\ast(B) = O \left(\frac{B^n}{\xi_1 \log \xi_1} + B^{n-1} \right).\]
\end{lemma}

Note that Lemma \ref{eke consequence} completes the proof of Lemma \ref{Main Lem}.

\section{The Selberg sieve}
\label{Selberg}

In this section we use a variant of the Selberg sieve to give an upper bound for $N_3(B)$. Our main goal in this section is to establish the following proposition:

\begin{proposition} \label{N3 estimate} Let $N_3(B)$ be as given in (\ref{M3}). Then $N_3(B) = o(B^n)$. 
\end{proposition}

Proposition \ref{N3 estimate} will follow from Lemmas \ref{H bound}, \ref{N3 smooth bound}, \ref{N4B}, \ref{QB lem 1}, and \ref{final lemma} below as well as Lemma \ref{eke consequence}. Consider the set 
\begin{equation} \label{final error} \N_3^\dagger(B) = \left\lbrace \Bx \in \bZ^n \cap [-B,B]^n \text{ } \bigg\vert
\begin{array}{@{}c@{}c@{}c}
F(\Bx) = uq^k, u \text{ is indivisible by } p^k \text{ for } p \leq \xi_1, \\
 \mbox{indivisible by } p^2 \text{ for } \xi_1 < p \leq \xi_2, \\
  q \text{ is a prime exceeding } \xi_2, \Bx \not \in \S_p \text{ for all } p | u.
\end{array}
\right\rbrace, 
\end{equation}
and put $N_3^\dagger(B) = \# \N_3^\dagger(B)$. Observe that
\begin{equation} \label{N3 decomp} N_3(B) = N_3^\dagger(B) + N_3^\ast(B).\end{equation} 
We shall establish the following preliminary result: 

\begin{lemma} \label{u factor} Let $\Bx \in \N_3^\dagger(B)$ and $u,q$ be as in (\ref{final error}). Then we have
\[u = O\left(B^2 (\log B)^{-2k/3} \right).\]
Furthermore, $u$ can be written as $u = u_1 u_2$, where $u_1$ divides
\[C(\xi_1) = \prod_{p \leq \xi_1} p^{k-1},\]
and $u_2$ is square-free with each prime divisor $p$ of $u_2$ satisfying $\xi_1 < p \leq \xi_2$. 
\end{lemma}

\begin{proof}
Observe that from $F(\Bx) = uq^k$ and our assumptions on $q$, we have
\[u = O\left(B^d \xi_2^{-k} \right).\]
By (\ref{k-d rel}) and (\ref{xi2}), there exists an absolute positive constant $C_1$ such that
\begin{align*} |u| & < C_1 B^{d - kn} (\log B)^{-2k/3} \\
& \leq C_1 B^{d - d + 2} (\log B)^{-2k/3} \\
& = C_1 B^2 (\log B)^{-2k/3}.
\end{align*}
We now factor $u$ into two factors $u_1$ and $u_2$, where $u_1$ consists of only prime factors less than $\xi_1$. We observe that since we have accounted for small prime powers via our treatment of $N_1(B)$, we have that $u_1$ divides $\prod_{p \leq \xi_1} p^{k-1}.$  The factor $u_2$, then, will be composed of prime factors larger than $\xi_1$. Further, it must be \emph{square-free}. This is because, by definition, the prime factors of $u$ between $\xi_1$ and $\xi_2$ divide $u$ exactly once, and $u$ cannot have a prime factor exceeding $\xi_2$, since otherwise
\[uq^k \gg B^{n(k+1)} \log B \gg B^d \log B,\]
which contradicts $\Bx \in [-B,B]^n$ for $B$ sufficiently large. \end{proof}
For each square-free integer $u_2$ such that each prime divisor $p$ of $u_2$ satisfies $\xi_1 < p \leq \xi_2$, put
\begin{equation} \label{big D}\D(u_2) = \prod_{\substack{\xi_1 < p \leq \frac{1}{12} \log(B^2 u_2^{-1}) \\ p \nmid u_2 \\ p \equiv 1 \pmod{k}}} p.\end{equation} 
We then have the following lemma:
\begin{lemma} \label{big D lem} Let $u_2$ be a square-free integer such that all of its prime divisors are between $\xi_1$ and $\xi_2$. Let $\omega(m)$ denote the number of distinct prime divisors of $m$. Let $\D(u_2)$ be as in (\ref{big D}). If $q > \xi_2$ is a prime, then there exists exactly $k^{\omega(\D)}$ residue classes $\{\fd_1, \cdots, \fd_{k^{\omega(D)}} \}$ such that
\[\fd_j^k \equiv q^k \pmod{\D}\]
for $j = 1, \cdots, k^{\omega(D)}$. 
\end{lemma} 

\begin{proof} Since all prime divisors of $\D$ are $O(\log B)$, it follows that $q^k$ is a proper $k$-th power residue modulo $\D$. Now consider the family of all $k$-th power residues modulo $\D$. By our choice of $\D$, we have that $k | \varphi(\D)$, so that the family of $k$-th power residues is not the set of all residues modulo $\D$. For each $p | \D$, $q^k$ has $k$ pre-images modulo $p$, meaning there exist $k$ distinct elements $\mathfrak{q}_1, \cdots, \mathfrak{q}_k$ in $\{0, 1, \cdots, p-1\}$ such that $\mathfrak{q}_j^k \equiv q^k \pmod{\Q}$. For a positive integer $l$ let us write $\omega(l)$ for the number of distinct prime divisors of $l$. Then it follows from the Chinese Remainder Theorem that there exist $k^{\omega(\D)}$ residue classes $\{\mathfrak{d}_1, \cdots, \mathfrak{d}_{k^{\omega(\D)}}\}$ modulo $\D$ such that $\mathfrak{d}_j^k \equiv q^k \pmod{\D}$. \end{proof}

Let $C_1$ be as in Lemma \ref{u factor}, and put $\xi_3 = C_1 B^2 (\log B)^{-2k/3}$. Lemmas \ref{u factor} and \ref{big D lem} have the following consequence, which is crucial for our estimation of $N_3(B)$:

\begin{lemma} \label{H bound} Let $u_1$ be a divisor of $C(\xi_1)$ and let $u_2$ a square-free integer whose prime divisors $p$ satisfy $\xi_1 < p \leq \xi_3$. Let $H_{u_1, u_2}(B)$ be the number of solutions $(m_1, \cdots, m_n) \in \bZ^n \cap [-B,B]^n$ to the following three congruences:
\begin{equation} \label{l-1} F(m_1, \cdots, m_n) \equiv 0 \pmod{u_1},
\end{equation}
\begin{equation} \label{l-2} F(m_1, \cdots, m_n) \equiv 0 \pmod{u_2},
\end{equation}
and for $0 \leq s < \D$, the solutions to the congruences
\begin{equation} \label{script D} F(m_1, \cdots, m_n) \equiv u_1 u_2 s^k \pmod{\D}
\end{equation}
such that $(m_1, \cdots, m_n) \pmod{p} \not \in \S_p$ for $p | u_1 u_2$. Then we have 
\begin{equation} \label{k-cover} N_3(B) \leq \sum_{\substack{u_1 | C(\xi_1) \\ u_2 \leq \xi_3}} \frac{ H_{u_1, u_2}(B)}{k^{\omega(\D)}} + N_3^\ast(B).\end{equation}
\end{lemma} 
\begin{proof} (\ref{k-cover}) follows from the fact that the solutions to (\ref{script D}) can be partitioned into sets of cardinality $k^{\omega(\D)}$ by Lemma \ref{big D lem}.  
\end{proof} 

In view of Lemma \ref{eke consequence}, we shall be primarily concerned with the term
\[N_3^\dagger(B) = \sum_{\substack{u_1 | C(\xi_1) \\ u_2 \leq \xi_3}} \frac{ H_{u_1, u_2}(B)}{k^{\omega(\D)}}.\]

\subsection{Selberg sieve weights} 

We now introduce the relevant Selberg sieve weights. Selberg devised an ingenious method to establish an upper bound for counting integer points in a box. To state this precisely, suppose that we wanted to count the set of points inside the box $[-B,B]^n$ satisfying a set of congruence conditions $\R_l$ modulo a positive integer $l$. Selberg introduced smooth functions $\gamma$ which satisfy the inequality 
\begin{equation} \label{majorant} \sum_{\substack{(m_1, \cdots, m_n) \in \bZ^n \cap [-B,B]^n \\ (m_1, \cdots, m_n) \in \R_l}} 1 \leq \sum_{\substack{(m_1, \cdots, m_n) \in \bZ^n \\ (m_1, \cdots, m_n) \in \R_l}} \gamma(m_1) \cdots \gamma(m_n),\end{equation}
where $\gamma$ is an upper bound for the characteristic function $\chi_B(z)$ of the interval $[-B,B]$, tends to zero rapidly outside of this interval, and is sufficiently smooth to be conducive to Fourier analysis and the Poisson summation formula. This reduces various counting problems into a question about exponential sums, from which one can draw results from a vast literature, including the seminal works of Weil and Deligne. \\ \\
Our choice of $\gamma$ is identical to that of Hooley's in \cite{Hoo1}. Namely, we start with the function, first given by Beurling and later utilized by Selberg to establish the optimal general bound for the large sieve inequality:
\begin{equation} \label{beurling} \Beu(z) = \left(\frac{ \sin \pi z}{\pi}\right)^2 \left(\sum_{n=0} \frac{1}{(z-n)^2} - \sum_{n = -\infty}^{-1} \frac{1}{(z-n)^2} + \frac{1}{2z}\right).
\end{equation} 
For the interval $[-U,U]$ we construct the function
\[g_U(z) = \frac{1}{2} \left(\Beu(U - z) + \Beu(U + z)\right)\]
which has the property that it is non-negative and majorizes the characteristic function of $[-U,U]$ (see \cite{Vaa}). Further, it satisfies the important property that its Fourier transform $\hat{g}_U(t)$ satisfies
\begin{equation} \label{g-hat} \hat{g}_U(t) = \begin{cases} 2U+1 & \text{if } t = 0, \\ 0 & \text{if } |t| > U \end{cases}
\end{equation}
and
\[\lvert \hat{g}_U(t) \rvert \leq 2U+1.\]
We now define the function $\gamma$ as
\begin{equation} \label{smooth weight} \gamma(z) = g_1\left(\frac{z}{B}\right),\end{equation}
whence it follows that
\[\hat{\gamma}(t) = B \hat{g}_1(Bt). \]
It is clear that $\gamma(z) \geq \chi_B(z)$ for all real numbers $z$. Because of the smoothness of $\gamma$, we can evaluate the sum 
\[\sum_{\substack{(m_1, \cdots, m_n) \in \bZ^n \\ (m_1, \cdots, m_n) \in \R_l}} \gamma(m_1) \cdots \gamma(m_n)\]
via Poisson summation. We have the following lemma, which is standard:

\begin{lemma} \label{poisson} Let $l$ be a positive integer, and let $\R_l$ be a subset of $(\bZ/ l \bZ)^n$. Let $\gamma$ be as in (\ref{smooth weight}), and put
\[M_{\R_l}(B) = \sum_{\substack{(m_1, \cdots, m_n) \in \bZ^n \\ (m_1, \cdots, m_n) \in \R_l}} \gamma(m_1) \cdots \gamma(m_n).\]
Let 
\begin{equation} \label{ER} E_{\R_l}(t_1, \cdots, t_n; l) = \sum_{(a_1, \cdots, a_n) \in \R_l} e^{-2\pi i(a_1 t_1 + \cdots + a_n t_n)/l}. \end{equation}
Then
\begin{equation} \label{fourier} M_{\R_l}(B) = \frac{1}{l^n} \sum_{(t_1, \cdots, t_n) \in \bZ^n} \hat{\gamma} \left(\frac{t_1}{l}\right) \cdots \hat{\gamma} \left(\frac{t_n}{l} \right) E_{\R_l} (t_1, \cdots, t_n ; l).\end{equation}

\end{lemma}
\begin{proof} See \cite{Hoo2}. 
\end{proof}

We shall decompose $M_{\R_l}(B)$ into two terms, given by
\begin{equation} \label{ML} M_{\R_l}(B) = M_{\R_l}^+(B) + O\left(M_{\R_l}^{++}(B) \right),\end{equation}
where
\[M_{\R_l}^+(B) = \frac{1}{l^n} \left(\hat{\gamma}(0)\right)^n E_{\R_l}(0, \cdots, 0;l) = \frac{(3B)^n \# \R_l}{l^n}\]
and
\[M_{\R_l}^{++}(B) = \frac{B^n}{l^n} \sideset{}{'} \sum_{|t_i| \leq l/B} \lvert E_{\R_l}(t_1, \cdots, t_n ; l) \rvert,\]
where the symbol $\sideset{}{'} \sum$ denotes that the tuple $(0, \cdots, 0)$ had been omitted. We then have the following:

\begin{lemma} \label{N3 smooth bound} Let $l = u_1 u_2 \D$, where $u_1, u_2, \D$ are as in Lemma \ref{H bound}. Put $l = u_1 u_2 \D$, and let $\R_l =\R_{u_1 u_2 \D}$ denote the set of congruence classes modulo $l$ satisfying (\ref{l-1}), (\ref{l-2}), and (\ref{script D}). Then
\[N_3^\dagger(B) \leq \sum_{\substack{u_1 | C(\xi_1) \\ u_2 \leq \xi_3}} \frac{M_{\R_l}^+(B)}{k^{\omega(\D)}} + O \left( \sum_{\substack{u_1 | C(\xi_1) \\ u_2 \leq \xi_3}} \frac{M_{\R_l}^{++} (B)}{k^{\omega(\D)}} \right).\]
\end{lemma}
\begin{proof} This follows from (\ref{k-cover}), (\ref{majorant}) and (\ref{ML}). \end{proof}
We put
\begin{equation} \label{N4 def} N_4(B) =  \sum_{\substack{u_1 | C(\xi_1) \\ u_2 \leq \xi_3}} \frac{M_{\R_l}^+(B)}{k^{\omega(\D)}}.\end{equation}
Our next lemma gives us an estimate for $N_4(B)$:

\begin{lemma} \label{N4B} Let $u_1, u_2, \D, l, \R_l$ be as in Lemma \ref{N3 smooth bound} and $N_4(B)$ as in (\ref{N4 def}). Then there exists a positive number $C_4$ such that 
\[N_4(B) = O\left(\frac{B^n \exp(2(n+1)(k-1) \xi_1)}{(\log B)^{C_4/\log_3 B}}\right).\]
\end{lemma}
 
\begin{proof} Let $\R_{u_1}, \R_{u_2}, \R_{\D}$ denote respectively the congruence classes corresponding to (\ref{l-1}), (\ref{l-2}) and (\ref{script D}), respectively. By the Chinese Remainder Theorem it follows that
\[\# \R_l = \# \R_{u_1} \# \R_{u_2} \# \R_{\D}.\]
Since $u_1 | C(\xi_1)$, it follows that $u_1 \leq C(\xi_1)$. From its definition and the result of Rosser and Schoenfeld \cite{RS}, we see that
\[C(\xi_1) \leq \exp(2(k-1) \xi_1).\] 
For $\R_{u_1}$, we use the trivial bound $\# \R_{u_1} = O(u_1^n) =  O(\exp(2n(k-1)\xi_1))$. We have $\# \R_{u_2} = O( u_2^{n-1} \tau_F(u_2) \sigma_{-1/4}(u_2))$ by Lemma \ref{rho lemma}, since $u_2$ is square-free. Observe that $\gcd(u_1 u_2, \D) = 1$. By the theorem of Lang and Weil \cite{LW}, which states that the number of points over $\bF_p$, for a prime $p | \D$, on the variety defined by the congruence
\[F(x_1, \cdots, x_n) - u_1 u_2 q^k \equiv 0 \pmod{p},\]
is
\[p^n + O(p^{n - 1/2}). \]
Then 
\begin{equation} \label{kappa D} \# \R_\D = \prod_{p | \D} \left(p^n + O\left(p^{n-1/2}\right)\right),\end{equation}
whence
\[\# \R_\D = \D^n \prod_{p | \D} \left(1 + O\left(p^{-1/2}\right) \right) = O(\D^n \sigma_{-1/4}(\D)).\]
Thus, by (\ref{ML}), (\ref{kappa D}), and Lemma \ref{rho lemma} we see that
\begin{align*} N_4(B) & = O\left( \exp(2n(k-1)\xi_1) \sum_{\substack{u_1 | C(\xi_1) \\ u_2 \leq \xi_3}} \frac{(3B)^n   u_2^{n-1}  \tau_F(u_2)\sigma_{-1/4}(u_2) \D^n  \sigma_{-1/4}(\D)}{(u_2 \D)^n k^{\omega(\D)}}\right)\\
& = O\left( \exp(2n(k-1) \xi_1) \sum_{\substack{u_1 | C(\xi_1) \\ u_2 \leq \xi_3}} \frac{B^n  \sigma_{-1/4}(u_2) \tau_F(u_2) \sigma_{-1/4}(\D)}{u_2 k^{\omega(D)}}\right).
\end{align*}
Observe that 
\begin{align*} \sigma_{-1/4}(\D) & = \prod_{p | \D} (1 + p^{-1/4}) \\
& = O\left(\left(\frac{2k}{3}\right)^{\omega(\D)} \right).
\end{align*}
It follows that
\[N_4(B) = O \left((\exp(2n(k-1) \xi_1) \sum_{\substack{u_1 | C(\xi_1) \\ u_2 \leq \xi_3}} \frac{B^n \tau_F(u_2) \sigma_{-1/4}(u_2)}{u_2 (3/2)^{\omega(D)}}\right).\]
Let us write 
\[\xi_4 = \xi_4(u_2) = \frac{1}{12} \log(B^2 u_2^{-1}),\]
and
\[\D' = \D'(u_2) = \prod_{p \leq \xi_4} p.\]
Observe that as $B^2 \xi_3^{-1} \rightarrow \infty$ as $B$ tends to infinity and $u_2^{-1} \gg \xi_3^{-1}$, we have
\[\log \D' = \sum_{p \leq \xi_4} \log p < \frac{12}{11} \xi_4\]
for $B$ sufficiently large, say by Rosser and Schoenfeld \cite{RS}. From (\ref{big D}), we see that
\[\D \leq \D' < \exp\left(12\xi_4/11\right) = \left(\frac{B^2}{u_2}\right)^{1/11}.\]
Next, we have
\[\omega(\D') = \pi(\xi_4; k, 1) \sim \frac{\xi_4}{\varphi(k) \log \xi_4},\]
where $\pi(B; q, a)$ is the counting function of primes $p$ satisfying $p \equiv a \pmod{q}$ up to $B$, and the above asymptotic follows from Dirichlet's theorem on primes in arithmetic progressions. Therefore we may find a constant $C_2$ such that
\[\omega(\D') > \frac{C_2 \xi_4}{\log \xi_4}\] 
for all $B$ sufficiently large. Observe that for a square-free number $l$, we have
\[\sigma_0(l) = \prod_{p | l} (1 + 1) = 2^{\omega(l)}.\]
From the definition of $\D$ and $\D'$, it follows that
\[(3/2)^{\omega(\D')} < (3/2)^{\omega(D')} C(\xi_1) (3/2)^{\gcd(\D', u_2)} < (3/2)^{\omega(\D)} C(\xi_1) \sigma_0(\gcd(\D', u_2).\]
Hence, there exists a positive number $C_3$ such that
\begin{equation} \label{pass to D prime} \frac{1}{(3/2)^{\omega(\D)}} < \frac{C_3}{(3/2)^{\omega(\D')}} \sigma_0(\gcd(\D', u_2)) \exp(2(k-1) \xi_1).
\end{equation}
From here we obtain the estimate
\begin{equation} \label{N4B pen} N_4(B) = O\left(\exp(2(n+1)(k-1)\xi_1) \sum_{\substack{ u_2 \leq \xi_3}} \frac{B^n \tau_F(u_2) \sigma_{-1/4}(u_2)\sigma_0(\gcd(\D', u_2)) }{(3/2)^{\omega(\D')} u_2} \right).\end{equation}
We now estimate the sum
\[S(t) = \sum_{u_2 \leq t} \tau_F(u_2) \sigma_{-1/4}(u_2) \sigma_0(\gcd(\D, u_2)).\]
We proceed, as with Hooley, by invoking his Lemma 6.2 in \cite{Hoo1}. We then have
\begin{align} S(t) & \leq \sum_{h | \D} \mu^2(h) \sigma_0(h) \sum_{\substack{u_2 \leq t \\ u_2 \equiv 0 \pmod{h}}} \tau_F(u_2) \sigma_{-1/4}(u_2) \\
& = \sum_{h | \D} \mu^2(h) \sigma_0(h) \sum_{\substack{u_2' h \leq t \\ \gcd(u_2', h) = 1}} \tau_F(h u_2') \sigma_{-1/4} (h u_2') \notag \\
& \leq \sum_{h | \D} \mu^2(h) \sigma_0(h) \tau_F(h) \sigma_{-1/4} (h) \sum_{u_2' \leq t/h} \tau_F(u_2') \sigma_{-1/4}(u_2') \notag \\
& = O \left(t \sum_{h | \D} \frac{\mu^2(h) \sigma_0(h) \tau_F(h) \sigma_{-1/4}(h)}{h} \right) \notag \\
& = O \left(t \prod_{w \leq \xi_4} \left(1 + \frac{2d+1}{w}\right) \right) \notag \\
& = O \left(t (\log \xi_4)^{2d+1}\right) \notag \\
& = O \left(t (\log \log B)^{2d+1}\right). \notag
\end{align}
By following Hooley's treatment of the term $N^{(6)}(X)$ in Section 8 of his paper \cite{Hoo1} and cutting the range of the summation in (\ref{N4B pen}) into dyadic parts, we see that, for some positive number $C_4$ we have
\[N_4(B) = O\left(\frac{B^n \exp(2(n+1)(k-1) \xi_1) }{(\log B)^{C_4/\log_3 B}}\right).\]
\end{proof}
We now put
\begin{equation} \label{gb} \xi_1(B) = \max\left\{1, \frac{C_4 \log \log B}{4(n+1)(k-1) \log_3 B}\right\}, \end{equation}
so that
\begin{align*} \frac{\exp(2(n+1)(k-1) g(B))}{\exp\left(C_4 \log_2 B/\log_3 B\right)} & = \exp\left(\frac{-C_4 \log_2 B}{2 \log_3 B} \right),
\end{align*}
whence
\[N_4(B) = O\left(B^n \exp\left(\frac{-C_4 \log_2 B}{2 \log_3 B}\right)\right) = o(B^n).\] 

Next we turn our attention to the much more difficult component
\begin{equation} \label{N5 def} N_5(B) = \sum_{\substack{u_1 | C(\xi_1) \\ u_2 \leq \xi_3}} \frac{M_{\R_l}^{++}(B)}{k^{\omega(\D)}}.\end{equation}
Recall from (\ref{ER}) that 
\[E_{\R_{l}}(t_1, \cdots, t_n; l) = E_{\R_{u_1}} E_{\R_{u_2}} E_{\R_{\D}}.\]
The term $E_{\R_{u_1}}(t_1, \cdots, t_n; u_1)$ can be trivially estimated by $u_1^n$, which is of size $O\left(\exp\left(\frac{C_4 \log_2 B}{4(k-1) \log_3 B} \right) \right)$. We now consider the term $E_{\R_{u_2}}$. For each prime $p$ dividing $u_2$ we write
\[F(\Bx) = F^\ast(\Bx) \prod_{j=1}^{\tau(p)} \L_i(\Bx),\]
where $F^\ast$ and $\L_i$ have coefficients in $\bZ_p$. We then write $E_{\R_{u_2}}$ as
\begin{align*} E_{\R_{u_2}}(t_1, \cdots, t_n; u_2) & = \prod_{p | u_2} \left(\sum_{1 \leq i \leq \tau_F(p)} \sum_{\substack{(a_1, \cdots, a_n) \in \bF_p^n \\ \L_i(a_1, \cdots, a_n) \equiv 0 \pmod{p}}} e^{2\pi i(a_1 t_1 + \cdots + a_n t_n)/p} \right) \\
& = \prod_{p | u_2} S(t_1, \cdots, t_n; p).
\end{align*}
We shall obtain the following estimate for $S(t_1, \cdots, t_n; p)$:
\begin{lemma} \label{S lemma} Let $p$ be a prime, and put
\[S(t_1, \cdots, t_n; p) = \sum_{1 \leq i \leq \tau_F(p)} \sum_{\substack{(a_1, \cdots, a_n) \in \bF_p^n \\ \L_i(a_1, \cdots, a_n) \equiv 0 \pmod{p}}} e^{2\pi i(a_1 t_1 + \cdots + a_n t_n)/p} .\]
Then we have
\begin{equation} \label{S-est} S(t_1, \cdots, t_n; p) \begin{cases} \leq \tau_F(p) p^{n-1}, & \text{if } t_1 x_1 + \cdots + t_n x_n \text{ divides } F(\Bx) \text{ over } \bF_p, \\
= 0, & \text{otherwise.} \end{cases}\end{equation}
\end{lemma} 
\begin{proof} We consider two scenarios. Suppose that
\[\L_s(x_1, \cdots, x_n) = \upsilon_1^{(s)} x_1 + \cdots + \upsilon_n^{(s)} x_n, \upsilon_j^{(s)} \in \bZ_p \text{ for } 1 \leq j \leq n.\]
If $(t_1, \cdots, t_n) \equiv \lambda (\upsilon_1^{(s)}, \upsilon_2^{(s)}, \cdots, \upsilon_n^{(s)}) \pmod{p}$ for some $\lambda \in \bF_p^\ast,$ then 
\[  \sum_{\substack{(a_1, \cdots, a_n) \in \bF_p^n \\ \L_s(a_1, \cdots, a_n) \equiv 0 \pmod{p}}} e^{2\pi i(a_1 t_1 + \cdots + a_n t_n)/p} = p^{n-1}.\]
Observe that since $p \nmid \Delta(F)$, that there does not exist $1 \leq s \leq \tau_F(p)$ such that $p | \upsilon_j^{(s)}$ for all $j = 1, \cdots, n$. We may suppose, without loss of generality, that $\upsilon_1^{(s)} \not \equiv 0 \pmod{p}$. Suppose that $\Ba \in \bF_p^n$ is such that
\[\L_s(\Ba) \equiv 0 \pmod{p}.\]
It then follows that 
\[a_1 \equiv - \left(\upsilon_1^{(s)}\right)^{-1} \left(\upsilon_2^{(s)} a_2 + \cdots + \upsilon_n^{(s)} a_n \right).\]
This implies
\begin{equation} \label{exp sum} \sum_{\substack{(a_1, \cdots, a_n) \in \bF_p^n \\ \L_s(a_1, \cdots, a_n) \equiv 0 \pmod{p}}} e^{2\pi i(a_1 t_1 + \cdots + a_n t_n)/p} = \sum_{(a_2, \cdots, a_n) \in \bF_p^n} e^{2\pi i(a_2(t_2 - t_1 (\upsilon_1^{(s)})^{-1} \upsilon_2^{(s)}) + \cdots + a_n(t_n - t_1 (\upsilon_1^{(s)})^{-1} \upsilon_n^{(s)}))/p}.\end{equation}
The right hand side can be written as
\[\prod_{j=2}^n \sum_{a_j \in \bF_p} e^{2\pi i a_j(\upsilon_1^{(s)} t_j - t_1 \upsilon_j^{(s)})/p}.\]
For each $j$, the sum
\[\sum_{a_j \in \bF_p} e^{2\pi i a_j(\upsilon_1^{(s)} t_j - t_1 \upsilon_j^{(s)})/p}\]
is zero unless the exponent is identically zero. This shows that (\ref{exp sum}) is non-zero if and only if $\upsilon_1^{(s)} t_j \equiv t_1 \upsilon_j^{(s)} \pmod{p}$ for $j = 2, \cdots, n$. This implies that
\begin{align*} (t_1, \cdots, t_n) & \equiv t_1 (\upsilon_1^{(s)})^{-1} (\upsilon_1^{(s)}, \upsilon_1^{(s)} t_2 t_1^{-1}, \cdots, \upsilon_1^{(s)} t_n t_1^{-1}) \pmod{p} \\
& \equiv t_1 (\upsilon_1^{(s)})^{-1}(\upsilon_1^{(s)}, \upsilon_2^{(s)}, \cdots, \upsilon_n^{(j)}) \pmod{p},\end{align*}
hence the first situation is the only case where the sum
\[\sum_{\substack{(a_1, \cdots, a_n) \in \bF_p^n \\ \L_s(a_1, \cdots, a_n) \equiv 0 \pmod{p}}} e^{2\pi i(a_1 t_1 + \cdots + a_n t_n)/p}\]
is non-zero. In other words, we have
\[S(t_1, \cdots, t_n; p) \begin{cases} \leq \tau_F(p) p^{n-1}, & \text{if } t_1 x_1 + \cdots + t_n x_n \text{ divides } F(\Bx) \text{ over } \bF_p, \\
= 0, & \text{otherwise,} \end{cases}\]
as desired. \end{proof}
For square-free $l$, let us write
\[S(t_1, \cdots, t_n; l) = \prod_{p | l} S(t_1, \cdots, t_n; p).\]
We have the following lemma:
\begin{lemma} Let $u_1, u_2, \D$ be as in Lemma \ref{N3 smooth bound}. Then 
\begin{equation} \label{MLL} \end{equation}
\begin{align*} \sum_{\substack{u_1 | C(\xi_1) \\ u_2 \leq \xi_3}} \frac{M_{l}^{++}(B)}{k^{\omega(\D)}} & = O\left(\exp(2(n+1)(k-1)\xi_1) \sum_{u_2 \leq \xi_3} \frac{B^n}{u_2^n} \sideset{}{'} \sum_{|t_1|, \cdots, |t_n| \leq l/B} S(t_1, \cdots, t_n; u_2)\right) 
\end{align*}
\end{lemma} 

\begin{proof} Recall that
\[M_l^{++}(B) = \frac{B^n}{u_1^n u_2^n \D^n} \sideset{}{'} \sum_{|t_i| \leq l/B} \lvert E_{\R_l} (t_1, \cdots, t_n; l) \rvert. \]
Note that
\[\lvert E_{\R_l}(t_1, \cdots, t_n; \D) \rvert = O \left(\D^n \sigma_{-1/4}(\D) \right),\]
and the multiplicativity of $E_{\R_l}$ implies that
\[\left \lvert E_{\R_l} (t_1, \cdots, t_n; l) \right \rvert = O \left(u_1^n \D^n \sigma_{-1/4}(\D) S(t_1, \cdots, t_n; u_2) \right).\]
Next note that 
\[\sigma_{-1/4}(\D) = O \left(k^{\omega(\D)}\right) \]
since $k \geq 2$. This then implies (\ref{MLL}), since the number of divisors of $C(\xi_1)$ does not exceed $C(\xi_1)$. 
\end{proof}

We now assess $S(t_1, \cdots, t_n; u_2)$ for an $n$-tuple $(t_1, \cdots, t_n) \in \bZ^n$. By Lemma \ref{S lemma}, this is zero unless for each prime $p | u_2$ there exists $\lambda_p \in \bF_p$ and $1 \leq s_p \leq \tau_F(p)$ such that $(t_1, \cdots, t_n) \equiv \lambda_p (\upsilon_1^{(s_p)}, \upsilon_2^{(s_p)}, \cdots, \upsilon_n^{(s_p)}) \pmod{p}$. One checks at once that for a fixed vector $\bup = (\upsilon_1, \cdots, \upsilon_n)$, the set
\[\{(x_1, \cdots, x_n) \in \bZ^n : (x_1, \cdots, x_n) \equiv \lambda (\upsilon_1, \cdots, \upsilon_n) \pmod{p} \text{ for some } \lambda \in \bF_p\} \]
is a lattice. For each prime $p$ dividing $u_2$, there are $\tau_F(p) \leq d$ such lattices to consider. If $(t_1, \cdots, t_n) \in \bZ^n$ is such that $S(t_1, \cdots, t_n; u_2)$ is non-zero, then it must lie on one such lattice for each prime divisor of $u_2$. Therefore, $(t_1, \cdots, t_n)$ lies on one of at most $d^{\omega(u_2)}$ lattices, each with determinant $u_2^{n-1}$. Let $\mathfrak{L}(u_2)$ denote the set of lattices for which the $n$-tuples $(t_1, \cdots, t_n)$ such that $S(t_1, \cdots, t_n; u_2) \ne 0$ are restricted to. \\ 

We now replace the bound $l/B$ for the variables $t_i$ in Lemma \ref{MLL} by something that is easier to work with. Observe that
\[u_1 \D = O\left(\exp(2(k-1) \xi_1) \left(\frac{B^2}{u_2} \right)^{1/11}\right).\]
Therefore, it follows that
\begin{equation} \frac{l}{B} = \frac{u_1 u_2 \D}{B} = O \left(\exp(2(k-1)\xi_1) \frac{B^{2/11}}{u_2^{1/11}} \frac{u_2}{B}\right)\end{equation}
\[= O \left(\exp(2(k-1)\xi_1) \left(\frac{u_2^{10/11}}{B^{9/11}} \right)\right).\]
Moreover, we have
\begin{equation} \label{major} \exp(2(k-1)\xi_1) \frac{u_2^{10/11}}{B^{9/11}} = O \left(\frac{u_2^{9/10}}{B^{4/5}}\right),
\end{equation} 
since 
\[\frac{u_2^{9/10}}{B^{4/5}} \cdot \frac{B^{9/11}}{u_2^{10/11}} = \left(\frac{B^2}{u_2}\right)^{1/55} \gg (\log B)^{\frac{2k}{165}} \gg (\log B)^{\frac{C_4}{2(n+1) \log_3 B}}.\]
Put
\begin{equation} \label{QB} Q(B) = \sum_{u_2 \leq \xi_3} \frac{1}{u_2^n} \sideset{}{'} \sum_{|t_1|, \cdots, |t_n| \leq u_2^{9/10}/B^{4/5}} S(t_1, \cdots, t_n; u_2).\end{equation}
Then it is clear that 
 \begin{equation} \label{MLL and Q} \sum_{\substack{u_1 | C(\xi_1) \\ u_2 \leq \xi_3}} \frac{M_{l}^{++}(B)}{k^{\omega(\D)}} = O(B^n g(B)^{k-1} Q(B)).\end{equation}
We shall assess $Q(B)$ by restricting the range of $u_2$ to a dyadic interval of the form $(U/2, U]$, with $U \leq \xi_3$. Denote this contribution to $Q(B)$ by $Q_U(B)$. We have the following lemma:

\begin{lemma} \label{QB lem 1} Let $Q(B)$ be as in (\ref{QB}). Then there exists a positive number $C_5$ such that for all $U > 1$, we have
\[Q_U(B) = O \left(\frac{U^{9/10} (\log B)^{C_5}}{B^{8/5}} \right).\]
\end{lemma}

\begin{proof} 
Let us write $F_s(x_1, x_s)$ for the product
\[F_s(x_1, x_s) = \prod_{j=1}^d (\psi_1^{(j)} x_1 + \psi_s^{(j)} x_s),\]
where $\psi_s^{(j)}$ are as in (\ref{C factor}). Note that each $F_s$ has integer coefficients. Moreover, since $F$ is irreducible over $\bQ$ it follows that each $F_s$ is a perfect power of a binary form with integer coefficients. Further, $F_s$ is not identically zero for $s = 2,\cdots,n$. If we fix a vector $(t_1, \cdots, t_n) \in \bZ^n$, then there are only at most $\sigma_0(F_2(t_2, -t_1))$ many $u_2$ such that $(t_1, \cdots, t_n) \in \Lambda$ for some $\Lambda \in \mathfrak{L}(u_2)$. To see this, if $(t_1, \cdots, t_n) \in \Lambda$ for $\Lambda \in \mathfrak{L}(u_2)$, then for each prime $p | u_2$, we have $(t_1, \cdots, t_n) \equiv \lambda_p (1, \upsilon_2^{(s)}, \cdots, \upsilon_n^{(s)}) \pmod{p}$ for some $\lambda_p \in \bF_p$ and $1 \leq s \leq \tau_F(p)$. Then it follows that $t_2 \equiv t_1 \upsilon_2^{(s)} \pmod{p}$, hence it follows that 
\[F_2(t_2, -t_1) \equiv 0 \pmod{p}.\]
This implies that $u_2 | F_2(t_2, -t_1)$, as claimed. Further, by the same argument we get that $u_2 | F_s(t_s, -t_1)$ for all $2 \leq s \leq n$. \\ \\
Now we can estimate $Q_U(B)$ when $U$ is suitably small as follows:
\begin{align*} Q_U(B) & \leq \frac{2^n}{U^n} \sum_{U/2 < u_2 \leq U} \sideset{}{'} \sum_{|t_1|, \cdots, |t_n| \leq U^{9/10}/B^{4/5}} S(t_1, \cdots, t_n; u_2) \\
& \leq \frac{2^n}{U} \sum_{U/2 < u_2 \leq U} d^{\omega(u_2)} \sideset{}{'} \sum_{\substack{|t_1|, \cdots, |t_n| \leq U^{9/10}/B^{4/5} \\ u_2 | \gcd(F_2(t_2, -t_1), \cdots, F_n(t_n, -t_1))}} 1.
\end{align*} 
Observe that when $t_1, t_2$ are fixed, then the condition $u_2 | F_j(t_j,-t_1)$ constrains each $t_j, j = 3, \cdots, n$ to at most $d^{\omega(u_2)}$ congruence classes modulo $u_2$, and for each congruence class, at most $(2U^{9/10}B^{-4/5})/u_2 + 1$ choices in the range $[-U^{9/10}/B^{4/5}, U^{9/10}/B^{4/5}]$. Since $U/2 < u_2 \leq U$, there is at most one choice when $B$ is sufficiently large. By the binomial theorem, for a number $A$ and a square-free positive integer $m$, we have
\[\sum_{r | m} A^{\omega(r)} = (A+1)^{\omega(m)}.\]
By permuting the variables if necessary, we may assume that $t_1 \ne 0$, at the cost of a factor of $n$. Hence
\begin{align} \label{QU bound} Q_U(B) & \leq \frac{n2^n}{U}  \sum_{\substack{|t_1|, |t_2| \leq U^{9/10}/B^{4/5}\\ t_1 \ne 0}} \sum_{u_2 | F_2(t_2, -t_1)} d^{(n-1)\omega(u_2)}\\ 
& = \frac{n2^n}{U} \sideset{}{'} \sum_{\substack{|t_1|, |t_2| \leq U^{9/10}/B^{4/5} \\ t_1 \ne 0}} (d^{n-1} + 1)^{\omega(F_2(t_2, -t_1))}, \notag 
\end{align}
so by Lemma 10.1 in \cite{Hoo1}, there exists a positive number $C_5$ such that
\[Q_U(B) = O\left(\frac{U^{9/10} (\log B)^{C_5}}{B^{8/5}}\right).\] 
\end{proof} 
If $U$ is relatively small, say $U < B^{5/3}$, then this is a satisfactory bound. Otherwise, we use Lemma 10.2 in \cite{Hoo1}, which we state as
\begin{lemma} \label{Hoo Lem 4} (Hooley, 2009) Set $\Xi(B) = B^{\frac{1}{6(\log \log B)^2}}$. Fix $u_2 \leq \xi_3$. Let $\omega^\dagger(m)$ denote the number of distinct prime factors of $m$ that exceed $\Xi$ and let 
\[l^\ast = \prod_{\substack{ p \leq \Xi \\ p | u_2 }} p\]
and
\[l^\dagger = \prod_{\substack{p > \Xi \\ p | u_2}} p.\]
Suppose that $l^\ast \leq B^{1/6}$. Then, for any positive constant $C_{6}$ and for $B^{1/2} < Y < B$, there exists a positive number $C_{7}$, depending only on $C_{6}$, such that
\[ \sideset{}{'}\sum_{\substack{(u_1, u_2) \equiv (t_1,t_2) \pmod{l^\ast} \\ |u_1, |u_2| \leq Y}} C_{6}^{\omega^\dagger(F(u_1, u_2))} = O\left(\frac{Y^2 (\log \log B)^{C_{7}}}{(l^\ast)^2} \right). \]
\end{lemma}
When $U > B^{5/3}$ we employ the divisors $l^\ast, l^\dagger$ of $u_2$ as in Lemma \ref{Hoo Lem 4}. Suppose firstly that $l^\ast > B^{1/6}$. This means that
\[B^{1/6} < \Xi^{\omega(l^\ast)} \leq \Xi^{\omega(u_2)},\]
which shows that
\[\omega(u_2) > (\log_2 B)^2.\]
Hence, either $\omega(u_2) > (\log_2 B)^2$ or $l^\ast \leq B^{1/6}$. Put
\begin{equation} \label{QU1} Q_U^{(1)}(B) = \sum_{\substack{U/2 < u_2 \leq U \\ \omega(u_2) > (\log_2 B)^2}} \frac{1}{u_2^n} \sideset{}{'} \sum_{|t_1|, \cdots, |t_n| \leq u_2^{9/10}/B^{4/5}} S(t_1, \cdots, t_n; u_2)  \end{equation}
and
\begin{equation} \label{QU2} Q_U^{(2)}(B) = \sum_{\substack{U/2 < u_2 \leq U \\ l^\ast \leq B^{1/6}}} \frac{1}{u_2^n} \sideset{}{'} \sum_{|t_1|, \cdots, |t_n| \leq u_2^{9/10}/B^{4/5}} S(t_1, \cdots, t_n; u_2).  \end{equation}

We have the following estimates for $Q_U^{(1)}(B)$ and $Q_U^{(2)}(B)$:

\begin{lemma} \label{final lemma} Let $Q_U^{(1)}(B), Q_U^{(2)}(B)$ be as in (\ref{QU1}) and (\ref{QU2}) respectively. Then there exists a positive number $C_6$ depending only on $d,n$ such that
\[Q_U^{(1)}(B) = O_d\left( \frac{U^{4/5} (\log B)^{C_{6}}}{B^{8/5} (\log B)^{\log_2 B}} \right) \]
and
\[Q_U^{(2)}(B) = O_d\left(\frac{ U^{4/5} \log B (\log_2 B)^{C_{7}}}{ B^{8/5}}\right). \] 
\end{lemma}
\begin{proof}
To estimate $Q_U^{(1)}(B)$, by (\ref{QU bound}) we have
\begin{align*} Q_U^{(1)}(B) & \leq \frac{n2^{n}}{U} \sum_{\substack{u_2 \leq U \\ \omega(u_2) > (\log \log B)^2}} \sideset{}{'} \sum_{|t_1|, |t_2| \leq u_2^{9/10}/B^{4/5}} d^{(n-1)\omega(F(t_2, -t_1))}  \\
& \ll_n \frac{1}{U} \sideset{}{'} \sum_{|t_1|, |t_2| \leq U^{9/10} / B^{4/5}} \sum_{\substack{u_2 | F_2(t_2, -t_1) \\  \omega(u_2) > (\log_2 B)^2}} d^{(n-1)\omega(u_2)}. \end{align*}
Observe that since $u_2 | F_2(t_2, -t_1),$ we have
\[d^{(n-1)\omega(u_2)} = \frac{d^{(n-1)\omega(u_2)} e^{(\log_2 B)^2}}{(\log B)^{\log_2 B}} < \frac{(3d^{n-1})^{\omega(u_2)}}{(\log B)^{\log_2 B}}.\]
By the binomial theorem and the fact that $u_2$ is square-free, it follows that
\[\sum_{\substack{u_2 | F_2(t_2, -t_1) \\ \omega(u_2) > (\log_2 B)^2}} d^{(n-1)\omega(u_2)} \leq \sum_{\substack{u_2 | F_2(t_2, -t_1) \\ \omega(u_2) > (\log_2 B)^2}} \frac{(3d^{n-1})^{\omega(u_2)}}{(\log B)^{\log_2 B}} = \frac{(3d^{n-1}+1)^{\omega(F_2(t_2,-t_1))}}{(\log B)^{\log_2 B}}.  \]
Hence, we see that for some positive $C_{6}$
\begin{align} \label{Q-nu-1}
Q_U^{(1)}(B) & \ll_n \frac{1}{U(\log B)^{\log_2 B}} \sideset{}{'} \sum_{|t_1|, |t_2| \leq U^{9/10}/B^{4/5}} (3d^{n-1} + 1)^{\omega(F_2(t_2, -t_1))} \\
& = O\left( \frac{U^{4/5} (\log B)^{C_{6}}}{B^{8/5} (\log B)^{\log_2 B}} \right) \notag \end{align}
by Lemma 10.1 in \cite{Hoo1} again. This completes the estimation of $Q_U^{(1)}(B)$. \\ 

Observe that 
\[\frac{U^{4/5}}{B^{8/5}} = O(\xi_3^{4/5} B^{-8/5}) = O((\log B)^{-8k/15}), \]
and thus the desired conclusion for $Q_U^{(1)}(B)$ holds. \\

The sum $Q_U^{(2)}(B)$ is more difficult. The key tool will be Lemma \ref{Hoo Lem 4}. Recall that $Q_U^{(2)}(B)$ consists of the contribution from those tuples for which $l^\ast \leq B^{1/6}$ and $U > B^{5/3}$. By the multiplicativity of $S(t_1, \cdots, t_n, \cdot)$, it follows that
\begin{equation} \label{Q-nu-2}Q_U^{(2)}(B)  \leq \frac{2^n}{U^n} \sideset{}{'}\sum_{|t_1|, \cdots, |t_n| \leq U^{9/10}/B^{4/5}} \sum_{\substack{l^\ast l^\dagger  \leq U \\ l^\ast \leq B^{1/6}}} S(t_1, \cdots,  t_n; l^\ast) S(t_1, \cdots, t_n; l^\dagger). \end{equation}
We rearrange the summation to obtain
\begin{equation} \label{QU22} \frac{2^n}{U^n} \sum_{\substack{l^\ast l^\dagger \leq U \\ l^\ast \leq B^{1/6}}} \sum_{b_1, \cdots, b_n \pmod{l^\ast}} S(b_1, \cdots, b_n; l^\ast) \sideset{}{'} \sum_{\substack{|t_1|, \cdots, |t_n| \leq U^{9/10}/B^{4/5} \\ t_i \equiv b_i \pmod{l^\ast}}} S(t_1, \cdots, t_n; l^\dagger).
\end{equation}
We estimate $S(t_1, \cdots, t_n; l^\dagger)$ by $d^{\omega(l^\dagger)} (l^\dagger)^{n-1}$ when it is non-zero. Next we observe that from the proof of Lemma \ref{QB lem 1} that $S(t_1, \cdots, t_n; l^\dagger)$ is non-zero only if $l^\dagger$ divides $F_s(t_s, -t_1)$ for $s = 2, \cdots, n$. Since $U > B^{5/3}$ and $l^\ast \leq B^{1/6}$, it follows that $l^\dagger > B^{3/2}$. Therefore $(U^{9/10} B^{-4/5})/l^\dagger \ll B^{-1/2} (\log B)^{-2k/3}$. In other words, for sufficiently large $B$ and for fixed $t_1, t_2$, the congruence condition imposed by $l^\dagger$ leads to at most $d^{\omega(l^\dagger)}$ choices for $t_3, \cdots, t_n$ as before. It then follows that
\begin{align*} Q_U^{(2)}(B)  & \leq \frac{n2^n}{U^n}  \sum_{\substack{l^\ast \leq B^{1/6} \\ b_1, \cdots, b_n \pmod{l^\ast}}} S(b_1, \cdots, b_n; l^\ast) \sideset{}{'}\sum_{\substack{|t_1|, |t_2| \leq U^{9/10}/B^{4/5} \\ t_i \equiv b_i \pmod{l^\ast}}} \sum_{l^\dagger | F_2(t_2,-t_1)} d^{\omega(l^\dagger)} (l^\dagger)^{n-1}  \\
& \leq \frac{n2^n}{U} \sum_{\substack{l^\ast \leq B^{1/6} \\ b_1, \cdots, b_n \pmod{l^\ast}}} \frac{S(t_1, \cdots, t_n; l^\ast)}{(l^\ast)^{n-1}}  \sideset{}{'}\sum_{\substack{|t_1|, |t_2| \leq U^{9/10}/B^{4/5} \\ t_i \equiv b_i \pmod{l^\ast}}}  \sum_{l^\dagger | F_2(t_2,-t_1)} d^{\omega(l^\dagger)} \\
& \ll \frac{n2^n}{U} \sum_{\substack{l^\ast \leq B^{1/6} \\ b_1, \cdots, b_n \pmod{l^\ast} }} \frac{S(b_1, \cdots, b_n; l^\ast)}{(l^\ast)^{n-1}} \sideset{}{'}\sum_{\substack{|t_1|, |t_2| \leq U^{9/10}/B^{4/5} \\ t_i \equiv b_i \pmod{l^\ast}}}  (d^{n-1} +1)^{\omega^\dagger(F_2(t_2,-t_1))} \\
& \ll \frac{n2^n}{U} \sum_{\substack{l^\ast \leq B^{1/6} \\ b_1, \cdots, b_n \pmod{l^\ast} }} \frac{S(b_1, \cdots, b_n; l^\ast)}{(l^\ast)^{n-1}}   \sideset{}{'} \sum_{\substack{|t_1|, |t_2| \leq U^{9/10}/B^{4/5} \\ t_1 \equiv b_1  \pmod{l^\ast} \\ t_2 \equiv b_2 \pmod{l^\ast}}} (d^{n-1}+1)^{\omega^\dagger(F(t_2,-t_1))}.
\end{align*} 
Note that $U < \xi_3 = C_1 B^2 (\log B)^{-2k/3}$, whence $U^{9/10}/B^{4/5} < B$. Further our assumption of $U > B^{5/3}$ shows that $U^{9/10}/B^{4/5} > B^{7/10}$. Hence, the innermost sum is treatable by Lemma \ref{Hoo Lem 4}. We then have
\[Q_U^{(2)}(B) = O_d\left(\frac{U^{4/5} (\log_2 B)^{C_{7}}}{B^{8/5} } \sum_{l^\ast \leq B^{1/6}} \frac{1}{(l^\ast)^{n+1}} \sum_{b_1, \cdots, b_n \pmod{l^\ast}} S(b_1, \cdots, b_n; l^\ast) \right).\]
By the proof of Lemma \ref{S lemma}, we se that for each prime $p$ we have
\[\sum_{b_1, \cdots, b_n \pmod{p}} S(b_1, \cdots, b_n; p) = p \cdot \tau_F(p) p^{n-1} = \tau_F(p) p^n.\]
It thus follows from multiplicativity that for any squarefree $l$ we have
\[\sum_{b_1, \cdots, b_n \pmod{l}} S(b_1, \cdots, b_n; l) = \tau_F(l) l^n.\]
We then deduce that
\[\sum_{l^\ast \leq B^{1/6}} \frac{1}{(l^\ast)^{n+1}} \sum_{b_1, \cdots, b_n \pmod{l^\ast}} S(b_1, \cdots, b_n; l^\ast) \leq \frac{\tau_F(l^\ast)}{l^\ast}.\]
By Lemma 6.1 in \cite{Hoo1}, we then see that
\begin{align*} Q_U^{(2)}(B) & = O_d\left(\frac{ U^{4/5} (\log_2 B)^{C_{7}}}{ B^{8/5}} \sum_{l^\ast \leq B^{1/6}} \frac{\tau_F(l^\ast)}{l^\ast} \right) \\ 
& = O_d\left(\frac{ U^{4/5} (\log_2 B)^{C_{7}}}{ B^{8/5}} \prod_{p \leq B^{1/6}} \left(1 + \frac{\tau_F(p)}{p} \right) \right) \\
& = O_d\left(\frac{ U^{4/5} \log B (\log_2 B)^{C_{7}}}{ B^{8/5}}\right),
\end{align*}
as desired. \end{proof}
By summing over $Q_U(B), Q_U^{(1)}(B), Q_U^{(2)}(B)$ over dyadic ranges of $U$ up to $\xi_2$, we then see that
\begin{equation} \label{QUB} \sum_{1 \leq k \ll \log B} O_d \left(\frac{(B^{5/3}/2^k)^{9/10} (\log B)^{C_5}}{B^{8/5}} \right) = O \left(\frac{(\log B)^{C_5}}{B^{1/10}} \right), \end{equation}

\begin{equation} \label{QUB1} \sum_{1 \leq k \ll \log B} O_d \left(\frac{(\xi_3/2^k)^{4/5} (\log B)^{C_6}}{B^{8/5} (\log B)^{\log_2 B}} \right) = O \left((\log B)^{C_6 - 8k/15 - \log_2 B} \right),\end{equation}
and
\begin{equation} \label{QUB2} \sum_{1 \leq k \ll \log B} O_d \left(\frac{(\xi_3/2^k)^{4/5} \log B (\log_2 B)^{C_7}}{B^{8/5}} \right) = O_d \left(\frac{(\log_2 B)^{C_7}}{(\log B)^{(2k-3)/3}} \right).
\end{equation}
This shows that
\begin{equation} \label{QB bound} Q(B) = O_F \left(\frac{(\log_2 B)^{C_7}}{(\log B)^{(2k-3)/3}}\right) = o(1),\end{equation}
and by (\ref{MLL and Q}), (\ref{N5 def}), Lemma \ref{eke consequence} and Lemma \ref{N3 smooth bound} we see that
\[N_3(B) = o(B^n),\]
and this completes the proof of Theorem \ref{MT1}.

\end{document}